\documentclass[11pt,leqno]{amsart}

\usepackage{amssymb}
\usepackage{amsmath}
\usepackage{amsthm}
\usepackage{latexsym}
\usepackage{amscd, xypic}
\usepackage[dvipsnames]{xcolor}
\usepackage{tikz-cd} 
\tikzcdset{arrow style=tikz}

\usepackage{enumitem}
\setlist[itemize]{leftmargin=*}
\setlist[enumerate]{leftmargin=*}

\begin{document}

\newtheorem{thm}{Theorem}
\newtheorem{prop}[thm]{Proposition}
\newtheorem{lem}[thm]{Lemma}
\newtheorem{cor}[thm]{Corollary}
\newtheorem*{que}{Question}
\newtheorem*{ass}{Assumption}
\theoremstyle{definition}
\newtheorem{defi}[thm]{Definition}
\newtheorem{ex}[thm]{Example}
\newtheorem{rem}[thm]{Remark}

\newcommand{\ob}{\mbox{ob}}
\newcommand{\mor}{\mbox{mor}}
\newcommand{\iso}{\mbox{Iso}}
\newcommand{\id}{\mbox{id}}
\newcommand{\G}{\mathcal {G}}
\newcommand{\U}{\mathcal {U}}
\newcommand{\N}{\mathbb {N}}
\newcommand{\si}{\sigma}
\newcommand{\rh}{\rho}
\newcommand{\ta}{\tau}
\newcommand{\rmod}{$R$\mbox{-mod}}
\newcommand{\grmod}{\mbox{$G$-$R$-mod}}
\newcommand{\HOM}{{\rm HOM}}
\newcommand{\Hom}{\mbox{Hom}}
\newcommand{\END}{{\rm END}}
\newcommand{\Ker}{\mbox{Ker}}
\newcommand{\Image}{\mbox{Im}}
\newcommand{\m}{^{-1}}

\newcommand{\GR}{\mbox{$\G$-$R$}}
\newcommand{\AbG}{\mbox{$\mbox{Ab}_{\G}$}}
\newcommand{\Ab}{\mbox{Ab}}
\newcommand{\rmd}{\mbox{$R$-md}}
\newcommand{\mdr}{\mbox{md-$R$}}
\newcommand{\rmds}{\mbox{$R$-md-$S$}}
\newcommand{\smod}{\mbox{$S$-mod}}
\newcommand{\rmods}{\mbox{$R$-mod-$S$}}
\newcommand{\modr}{\mbox{mod-$R$}}
\newcommand{\gsmod}{\mbox{$\G$-$S$-mod}}
\newcommand{\gmodr}{\mbox{$\G$-mod-$R$}}
\newcommand{\grmods}{\mbox{$\G$-$R$-mod-$S$}}
\newcommand{\grmodr}{\mbox{$\G$-$R$-mod-$R$}}
\newcommand{\gmods}{\mbox{$\G$-mod-$S$}}
\newcommand{\grmd}{\mbox{$\G$-$R$-md}}
\newcommand{\gsmd}{\mbox{$\G$-$S$-md}}
\newcommand{\gmdr}{\mbox{$\G$-md-$R$}}
\newcommand{\gmds}{\mbox{$\G$-md-$S$}}
\newcommand{\grmds}{\mbox{$\G$-$R$-md-$S$}}

\newcommand{\rumod}{R\text{-\textup{\textbf{umod}}}}
\newcommand{\sumod}{S\text{-\textup{\textbf{umod}}}}
\newcommand{\modgr}{\mbox{{\bf mod}-}\G\mbox{-$R$}}
\newcommand{\grumod}{\G\textup{-}\rumod}
\newcommand{\gsumod}{\G\textup{-}\sumod}
\newcommand{\umodgr}{\G\text{-\textup{\textbf{umod}}}\textup{-}R}
\newcommand{\umodgs}{\G\text{-\textup{\textbf{umod}}}\textup{-}S}

\title[Chain conditions for rings with enough Idempotents]{Chain conditions for rings with enough idempotents with applications to category graded rings}

\author[P. Lundstr\"{o}m]{Patrik Lundstr\"{o}m}
\address{University West,
Department of Engineering Science, 
SE-46186 Trollh\"{a}ttan, Sweden}
\email{{\scriptsize patrik.lundstrom@hv.se}}



\begin{abstract}
We obtain criteria for when a ring with enough idempotents
is left/right artinian or noetherian in terms 
of local criteria defined by the associated 
complete set of idempotents for the ring.
We apply these criteria to object unital 
category graded rings in general and, in particular, to 
the class of skew category algebras.
Thereby, we generalize results by Nastasescu-van Oystaeyen,
Bell, Park and Zelmanov from the group graded case to 
groupoid, and in some cases category, gradings. 
\end{abstract}

\maketitle

\section{Introduction}\label{section:introduction}

Throughout this article all rings are associative
but not necessarily unital.
If a ring $S$ is unital, then 
we always assume that $S$ is nonzero
and that $1_S$ denotes the multiplicative 
unit of $S$.
Recall that a ring is called left/right 
\emph{artinian (noetherian)} if it satisfies the the 
descending (ascending) 
chain condition on its poset of left/right ideals.
The main purpose of the present article is to establish
criteria for 
artiniarity and noetherianity for rings with enough idempotents
(see Theorem~\ref{thm:main}). The 
secondary purpose is to apply this result
to the setting of category graded rings 
and thereby obtaining similar results for this class
of rings (see Theorems 
\ref{thm:generalizationnas}-\ref{thm:generalizationzelmanov}). 

Here is an outline of the article.
In Section \ref{sec:chainconditionsenough},
we analyze chain conditions for the class of rings with enough 
idempotents, introduced by Fuller in \cite{Fu}
(see Definition \ref{def:enoughidempotents}).
To this end, we consider of rings
with complete sets of idempotents $\{ e_i \}_{i \in I}$ 
which are,
what we call, \emph{strong} (see Definition \ref{def:strong}).
By this we mean that for all $i,j \in I$ with 
$e_i S e_j$ nonzero or $e_j S e_i$ nonzero,
then $e_i \in e_i S e_j S e_i$ and
$e_j \in e_j S e_i S e_j$.
We establish the following:

\begin{thm}\label{thm:main}
Suppose $S$ is a ring with enough idempotents and
let $\{ e_i \}_{i \in I}$ be a complete 
set of idempotents for $S$.
\begin{enumerate}

\item If $S$ is left/right artinian (noetherian),
then $I$ is finite and for every $i \in I$ the 
ring $e_i S e_i$ is left/right artinian (noetherian).

\item Suppose $\{ e_i \}_{i \in I}$ is strong. 
Then $S$ is left/right artinian 
(noetherian) if and only if
$I$ is finite and for every $i \in I$ the ring
$e_i S e_i$ is left/right artinian (noetherian).

\end{enumerate}
\end{thm}
In Section \ref{sec:homsetstrong},
we state our conventions on categories and 
groupoids (see Definition \ref{def:category}).
We also introduce the class of \emph{hom-set strong} categories
(see Definition \ref{def:homsetstrong}).
This is a class of categories which strictly 
contains the class of groupoids (see Example 
\ref{ex:notgroupoid}).

In Section \ref{sec:chaincategory}, we
consider the class of object unital category (and groupoid) 
graded rings 
(for our conventions on graded rings, see Definition \ref{def:categorygraded}).
We use Theorem \ref{thm:main} and results
from the group graded case by
Nastasescu-van Oystaeyen \cite{nastasescu2004} 
(see Theorem \ref{thm:nastasescu}) and 
Bell \cite{bell1987} (see Theorem \ref{thm:bell}),
to establish the following:

\begin{thm}\label{thm:generalizationnas}
Suppose $G$ is a groupoid and let 
$S$ be a ring which is strongly $G$-graded and object unital.
\begin{enumerate}

\item Let $G$ be torsion-free. Then $S$ is 
left/right artinian if and only if $G_0$
is finite and for every $a \in G_0$ the ring
$S_a$ is left/right artinian and $S_{G(a)}$
is finitely generated as a left/right $S_a$-module.

\item Let $G$ be polycyclic-by-finite. 
Then $S$ is left/right noetherian if and only if
$G_0$ is finite and for every $a \in G_0$
the ring $S_a$ is left/right noetherian.  

\end{enumerate}
\end{thm}
For the definitions of the classes of torsion-free
and polycyclic-by-finite groupoids, see
Definitions \ref{def:torsionfree} and 
\ref{def:polycyclicbyfinite}.

In Section \ref{sec:chainskew}, we analyse
chain conditions for skew category algebras
which are defined by skew category systems of unital rings
(see Definition \ref{def:skewcategory}).
Using a result by Park \cite{park1979}
(see Theorem \ref{thm:park})
from the setting of skew group rings,
and the results established in Section \ref{sec:chaincategory}, 
we prove the following: 

\begin{thm}\label{thm:generalizationpark}
Suppose $G$ is a groupoid and let
$\alpha = \{ \alpha_g : R_{d(g)} 
\to R_{c(g)} \}_{g \in G_1}$ be a skew category system
of unital rings. 
Then the associated skew category algebra
$R *_\alpha G$ is left/right artinian  
if and only if $G$ is finite and for 
for every $a \in G_0$ the ring $R_a$ is 
is left/right artinian.
\end{thm}

Note that Theorem \ref{thm:generalizationpark}
already was obtained in \cite{nystedt2018}
using a different method. 
In Section \ref{sec:chainskew},
we also show the following theorem,
using a result by Zelmanov \cite{zelmanov1977} 
(see Theorem \ref{thm:zelmanov})
and the results in Section \ref{sec:chaincategory}.

\begin{thm}\label{thm:generalizationzelmanov}
Suppose $G$ is a hom-set strong category and let
$T$ be a unital ring. Then 
the associated category ring $T[G]$ is left/right artinian 
if and only if $T$ is left/right 
artinian and $G$ is finite.
\end{thm}

Theorems 
\ref{thm:generalizationnas}-\ref{thm:generalizationzelmanov} generalize classical results from the group graded case to 
the groupoid (and in the case of the last theorem 
category) graded situation.

\section{Rings with enough idempotents}\label{sec:chainconditionsenough}

In this section, we analyze chain conditions 
for the class of rings with 
\linebreak
enough idempotents. 
To this end, we introduce the class of rings with 
\linebreak
complete sets of idempotents which are,
what we call, \emph{strong}.
At the end of this section, we prove Theorem \ref{thm:main}. 
Throughout this article, we put 
$\mathbb{N} := \{ 1,2,3,\ldots \}$.
For the rest of the section,
$R$ and $S$ denote rings.
For the convenience of the 
reader, we first gather some well known results that 
we will use later.

\begin{prop}\label{prop:module} 
The following assertions hold.
\begin{enumerate}

\item Suppose $M$ is a left/right $S$-module
and let $R \subseteq S$.
If $M$ is artinian (noetherian) as a left/right $R$-module, 
then $M$ is artinian 
(noetherian) as a 
left/right $S$-module.

\item Suppose $R$ is a direct summand of the left/right 
$R$-module $S$.
Then, for any right/left ideal $I$ of $R$,
the equality $IS \cap R = IR $ (or $SI \cap R = RI$) holds.
In particular, if $S$ is right/left artinian (noetherian) 
and $R$ is unital, then $R$ is right/left artinian (noetherian).

\item Suppose $n \in \mathbb{N}$ and let
$R_1,\ldots,R_n$ be rings.
Then $R_1 \times \cdots \times R_n$ is left/right
artinian (noetherian) if and only if all of the rings
$R_1,\ldots,R_n$ are left/right artinian (noetherian).

\item Suppose $R$ is unital and let
$M$ be a left (right) $R$-module.
Let $n \in \mathbb{N}$ and suppose 
$M_1,\ldots,M_n$ are left/right $R$-submodules of $M$
such that $M = M_1 \oplus \cdots \oplus M_n$.
Then $M$ is left/right artinian (noetherian) if and only if 
all of the modules $M_1,\ldots,M_n$ are
left/right artinian (noetherian).

\end{enumerate}
\end{prop}

\begin{proof}
See a standard book on the theory of rings and
modules e.g. \cite{rowen1998}.
\end{proof}

\begin{defi}\label{def:enoughidempotents}
Recall from \cite{Fu} that $S$ is said to have  
\emph{enough idempotents} if there exists a set 
$\{ e_i \}_{i \in I}$ of nonzero orthogonal idempotents in $S$, 
called a \emph{complete set of idempotents} for $S$, 
such that $S = \oplus_{i\in I} S e_i = 
\oplus_{i \in I} e_i S$.
For the rest of this section, 
$S$ denotes a ring with enough
idempotents where $\{ e_i \}_{i \in I}$ is a fixed complete set 
of idempotents for $S$. 
Given $i,j \in I$, we use the notation 
$S_{ij} := e_i S e_j$ and $S_i := S_{ii}$.
We also put $S_0 := \oplus_{i \in I} S_i$.
\end{defi}

\begin{prop}\label{prop:finite}
If $S$ is left/right artinian (noetherian), then $I$ is finite.
\end{prop}

\begin{proof}
Suppose $I$ is infinite. Then we may assume that 
$\mathbb{N} \subseteq I$.

Define a set of left ideals $\{ I_i \}_{i \in \mathbb{N}}$
of $S$ by $I_i = S e_i + S e_{i+1} + S e_{i+2} + \cdots$
for $i \in \mathbb{N}$.
Then $e_i \in I_i \setminus I_{i+1}$ and thus
$I_i \supsetneq I_{i+1}$ for all $i \in \mathbb{N}$.
Therefore $S$ is not left artinian. 

Define a set of right ideals $\{ J_i \}_{i \in \mathbb{N}}$
of $S$ by $J_i = e_i S + e_{i+1} S + e_{i+2} S + \cdots$
for $i \in \mathbb{N}$.
Then $e_i \in J_i \setminus J_{i+1}$ and thus
$J_i \supsetneq J_{i+1}$ for all $i \in \mathbb{N}$.
Therefore $S$ is not right artinian.

Define a set of left ideals 
$\{ K_i \}_{i \in \mathbb{N}}$ of $S$
by $K_i = S e_1 + S e_2 + \cdots + S e_i$
for $i \in \mathbb{N}$.
Then $e_{i+1} \in K_{i+1} \setminus K_i$ and thus
$K_i \subsetneq K_{i+1}$ for all $i \in \mathbb{N}$.
Therefore $S$ is not left noetherian.

Define a set of right ideals 
$\{ L_i \}_{i \in \mathbb{N}}$ of $S$
by $L_i = e_1 S + e_2 S + \cdots + e_i S$
for $i \in \mathbb{N}$.
Then $e_{i+1} \in L_{i+1} \setminus L_i$ and thus
$L_i \subsetneq L_{i+1}$ for all $i \in \mathbb{N}$.
Therefore $S$ is not right noetherian.
\end{proof}

\begin{prop}\label{prop:ringi}
Suppose $S$ is left/right artinian (noetherian)
and $i \in I$. Then the ring $S_i$ 
is left/right artinian (noetherian).
\end{prop}

\begin{proof}
From Proposition \ref{prop:finite} it follows that 
$S_0$ is unital with $1_{S_0} = \sum_{j \in I} e_j$.
Since $S_0$ is a direct summand
of $S$ as a right/left $S_0$-module it follows
from Proposition \ref{prop:module}(2) that $S_0$
is left/right artinian (noetherian).
By Proposition \ref{prop:module}(3) the
ring $S_i$ is left/right artinian (noetherian).
\end{proof}

\begin{prop}\label{prop:equivalentIII}
Given $S$ and $\{ e_i \}_{i \in I}$
the following properties are equivalent:
\begin{enumerate}

\item $\forall (i,j,k) \in I \times I \times I$ 
if two of the additive groups
$S_{ij}$, $S_{jk}$ and $S_{ik}$ are nonzero,
then the third one is also nonzero and 
$S_{ij} S_{jk} = S_{ik}$;

\item $\forall (p,q) \in I \times I$ 
if one of the additive groups $S_{pq}$
and $S_{qp}$ is nonzero, then the other one is also
nonzero and $S_{pq} S_{qp} = S_p$;

\item  $\forall (p,q) \in I \times I$ 
if one of the additive groups $S_{pq}$
and $S_{qp}$ is nonzero, then the other one is also
nonzero and $e_p \in S_{pq} S_{qp}$.

\end{enumerate}
\end{prop}

\begin{proof}
(1)$\Rightarrow$(2): 
Suppose (1) holds and take $(p,q) \in I \times I$.
\begin{itemize}
\item Case 1: $S_{pq} \neq \{ 0 \}$.
Since $S_p \neq \{ 0 \}$ it follows from (1),
with $i = p$, $j = q$ and $k = p$, that 
$S_{qp} \neq \{ 0 \}$ and $S_{pq} S_{qp} = S_p$.

\item Case 2: $G(q,p) \neq\{ 0 \}$.
Since $S_q \neq \{ 0 \}$ it follows from (1),
with $i = q$, $j = p$ and $k = q$, that 
$S_{pq} \neq \{ 0 \}$ and $S_{qp} S_{pq} = S_q$.
\end{itemize}
(2)$\Rightarrow$(1): Suppose (2) holds and 
take $(i,j,k) \in I \times I \times I$.
\begin{itemize}
\item Case 1: $S_{ij} \neq \{ 0 \}$ and
$S_{jk} \neq \{ 0 \}$. 
By (2) we get that
$\{ 0 \} \neq S_j = S_j S_j = 
S_{ji} S_{ij} S_{jk} S_{kj} \subseteq S_{ji} S_{ik} S_{kj}$.
Therefore $S_{ik} \neq \{ 0 \}$. 
By (2) again we get that
$S_{ik} = S_{ik} S_k = S_{ik} S_{kj} S_{jk}
\subseteq S_{ij} S_{jk} \subseteq S_{ik}$.
Thus $S_{ij} S_{jk} = S_{ik}$.

\item Case 2: $S_{ij} \neq \{ 0 \}$ and
$S_{ik} \neq \{ 0 \}$.
By (2) we get that
$\{ 0 \} \neq S_i = S_i S_i = 
S_{ij} S_{ji} S_{ik} S_{ki} \subseteq S_{ij} S_{jk} S_{ki}$.
Therefore $S_{jk} \neq \{ 0 \}$. 
The same calculation as in Case 1 shows that
$S_{ij} S_{jk} = S_{ik}$.

\item Case 3: $S_{ik} \neq \{ 0 \}$ and
$S_{jk} \neq \{ 0 \}$.
By (2) we get that
$\{ 0 \} \neq S_k = S_k S_k = 
S_{ki} S_{ik} S_{kj} S_{jk} \subseteq S_{ki} S_{ij} S_{jk}$.
Therefore $S_{jk} \neq \{ 0 \}$. 
The same calculation as in Case 1 shows that
$S_{ij} S_{jk} = S_{ik}$.
\end{itemize}
(2)$\Leftrightarrow$(3): This is clear since
$e_p \in S_p$ for all $p \in I$.
\end{proof}

\begin{defi}\label{def:strong}
If $S$ and $\{ e_i \}_{i \in I}$ satisfy any of the 
three equivalent properties 
\linebreak
in 
Proposition \ref{prop:equivalentIII}, then we 
say that $\{ e_i \}_{i \in I}$ is a \emph{strong}
complete set of 
\linebreak
idempotents for $S$.
\end{defi}

\begin{prop}\label{prop:lattice}
Let $\{ e_i \}_{i \in I}$ be a strong complete 
set of idempotents for $S$.  
Suppose $S_{ij} \neq \{ 0 \}$
for some $i,j \in I$.
\begin{enumerate}

\item The poset of left (right) ideals of 
$S_i$ ($S_j$) is isomorphic to the 
poset of left (right) $S_j$-submodules 
of $S_{ji}$ ($S_{ij}$).

\item The ring $S_i$ 
($S_j$) is left (right) artinian/noetherian
if and only if the left 
$S_j$-module (right $S_i$-module) $S_{ji}$
is artinian/noetherian.

\end{enumerate}
\end{prop}

\begin{proof}
We show the ``left'' part of the proof and leave the 
``right'' part to the reader. 
Since (2) obviously follows from (1) it is 
enough to show (1). 
\linebreak
Define maps
$
\alpha : \{ \mbox{left ideals of $S_i$} \} 
\to
\{ \mbox{left $S_j$-submodules of $S_{ji}$} \}
$ and
\linebreak
$
\beta : \{ \mbox{left $S_j$-submodules of $S_{ji}$} \}
\to
\{ \mbox{left ideals of $S_i$} \}
$
by $\alpha(I) = e_j S I$ and $\beta(M) = e_i S M$
for left ideals $I$ of $S_i$ and 
left $S_j$-submodules $M$ of $S_{ji}$.
Then, clearly, $\alpha$ and $\beta$ are 
well defined inclusion preserving maps. 
Take a left ideal $I$ of $S_i$ and a
left $S_j$-submodule $M$ of $S_{ji}$.
By Proposition \ref{prop:equivalentIII} we get that
$
\beta( \alpha (I) ) = \beta( e_j S I )
= S_{ij} S_{ji} I = S_i I = I
$ 
and
$
\alpha( \beta(M) ) = \alpha( e_i S M )
= S_{ji} S_{ij} M = S_j M = M
$.
This proves (1).
\end{proof}

\begin{prop}\label{prop:IF}
Let $\{ e_i \}_{i \in I}$ be a finite strong complete 
set of idempotents 
\linebreak
for $S$.
Suppose that 
for every $i \in I$ the ring $S_i$ is
left/right artinian 
\linebreak
(noetherian).
Then $S$ is left/right artinian (noetherian). 
\end{prop}

\begin{proof}
We show the ``left'' part of the proof and leave the 
``right'' part to the reader. 
Take $i,j \in I$. 
By Proposition \ref{prop:lattice}(2) it follows
that the left $S_j$-module $S_{ji}$ is
artinian (noetherian).
By Proposition \ref{prop:module}(1) it follows that 
$S_{ji}$ is artinian (noetherian) as a left $S_0$-module.
Since $I$ is finite, it follows from 
Proposition \ref{prop:module}(4) that 
$S = \oplus_{k,l \in I} S_{kl}$ is artinian
(noetherian) as a left $S_0$-module.
By Proposition \ref{prop:module}(1),
$S$ is left artinian (noetherian).
\end{proof}

\subsubsection*{Proof of Theorem \ref{thm:main}}
The ``only if'' statements follow from Proposition
\ref{prop:finite} and Proposition \ref{prop:ringi}.
The ''if'' statement follows from Proposition \ref{prop:IF}.
\qed

\section{Hom-set strong categories}\label{sec:homsetstrong}

In this section, we state our conventions on categories 
and groupoids. 
\linebreak
We introduce the class of 
hom-set strong categories and
we show that this
class strictly contains the class of groupoids.
At the end of the section, we show a finiteness
result for hom-set strong categories.

\begin{defi}\label{def:category}
For the rest of this article, unless otherwise stated, 
we let $G$ 
denote a small category.
Recall that this means that the 
collection of objects in $G$, denoted by $G_0$, 
and the collection of morphisms in $G$,
denoted by $G_1$, are sets.
The domain and codomain of $g \in G_1$
is denoted by $d(g)$ and $c(g)$ respectively;
we indicate this by writing $g : d(g) \to c(g)$.
Given $a,b \in G_0$ we let $G(a,b)$ denote the 
set of morphisms $b \to a$. We put $G(a) := G(a,a)$
for $a \in G_0$.
We always regard $G_0$ as a subset of $G_1$.
Therefore, we denote the identity morphism 
$a \to a$ by $a$.
We put $G_2 := \{ (g,h) \in G_1 \times
G_1 \mid d(g) = c(h) \}$.
If $(g,h) \in G_2$, then the composition of $g$ and $h$ 
is written as $gh$. 
We say that $G$ is finite if $G_1$, and hence also $G_0$,
is finite.
Recall that a category $G$ is said to be a \emph{groupoid}
if all morphisms in $G$ are isomorphisms. 
In that case, the inverse of a morphism $g : a \to b$
in $G$ is denoted by $g^{-1} : b \to a$.
\end{defi}

\begin{prop}\label{prop:equivalentcategory}
For a category $G$, the following properties are equivalent:
\begin{enumerate}

\item $\forall (a,b,c) \in G_0 \times G_0 \times G_0$ 
if two of the sets
$G(a,b)$, $G(b,c)$ and $G(a,c)$ are nonempty,
then the third set is nonempty and 
$G(a,b) G(b,c) = G(a,c)$;

\item $\forall (x,y) \in G_0 \times G_0$ 
if one of the sets $G(x,y)$
and $G(y,x)$ is nonempty, then the other set is 
nonempty and $G(x,y) G(y,x) = G(x)$;

\item $\forall (x,y) \in G_0 \times G_0$ 
if one of the sets $G(x,y)$
and $G(y,x)$ is nonempty, then the other set is 
nonempty and $x \in G(x,y) G(y,x)$.

\end{enumerate}
\end{prop}

\begin{proof}
(1)$\Rightarrow$(2): 
Suppose (1) holds and take $(x,y) \in G_0 \times G_0$.
\begin{itemize}
\item Case 1: $G(x,y) \neq \emptyset$.
Since $G(x,x) \neq \emptyset$ it follows from (1),
with $a = x$, $b = y$ and $c = a$, that 
$G(y,x) \neq \emptyset$ and $G(x,y)G(y,x) = G(x)$.

\item Case 2: $G(y,x) \neq \emptyset$.
Since $G(y,y) \neq \emptyset$ it follows from (1),
with $a = y$, $b = x$ and $c = y$, that 
$G(x,y) \neq \emptyset$ and $G(y,x)G(x,y) = G(y)$.
\end{itemize}
(2)$\Rightarrow$(1): Suppose (2) holds and 
take $(a,b,c) \in G_0 \times G_0 \times G_0$.
\begin{itemize}
\item Case 1: $G(a,b) \neq \emptyset$ and
$G(b,c) \neq \emptyset$. Then
$G(a,c) \supseteq G(a,b)G(b,c) \neq \emptyset$
so that $G(a,c) \neq \emptyset$.
From (2), with $x = c$ and $y = b$, we get that
$G(c,b) \neq \emptyset$ and
$G(c,b)G(b,c) = G(c)$. Thus
$
G(a,c) = G(a,c)G(c) = G(a,c)G(c,b)G(b,c)
\subseteq G(a,b)G(b,c) \subseteq G(a,c)
\Rightarrow G(a,b)G(b,c) = G(a,c)$.

\item Case 2: $G(a,b) \neq \emptyset$ and
$G(a,c) \neq \emptyset$.
From (2), with $x = a$ and $y = b$, it follows that
$G(b,a) \neq \emptyset$. Thus
$G(b,c) \supseteq G(b,a)G(a,c) \neq \emptyset$
and hence $G(b,c) \neq \emptyset$. 
A calculation similar to Case 1 gives
$G(a,b)G(b,c) = G(a,c)$.

\item Case 3: $G(a,c) \neq \emptyset$ and
$G(b,c) \neq \emptyset$.
From (2), with $x = c$ and $y = b$, it follows that
$G(b,c) \neq \emptyset$.
Thus $G(a,b) \supseteq G(a,c)G(c,b) \neq \emptyset$
and hence $G(a,b) \neq \emptyset$. 
A calculation similar to Case 1 gives
$G(a,b)G(b,c) = G(a,c)$.
\end{itemize}
(2)$\Leftrightarrow$(3): This is clear since
$x \in G(x)$ for all $x \in G_0$.
\end{proof}

\begin{defi}\label{def:homsetstrong}
If a category $G$ satisfies any of the equivalent properties
in Proposition \ref{prop:equivalentcategory},
then we say that $G$ is a \emph{hom-set strong} category.
\end{defi}

\begin{prop}\label{prop:groupoidimplieshom}
If $G$ is a groupoid, then $G$ is a hom-set strong 
category.
\end{prop}

\begin{proof}
This follows from Proposition \ref{prop:equivalentcategory}(3)
and the fact that for every 
\linebreak
morphism $g : b \to a$ in $G$ 
the relation $a = g g^{-1} \in G(a,b)G(b,a)$ holds. 
\end{proof}
 
\begin{ex}\label{ex:notgroupoid}
Not all hom-set strong categories are groupoids.
Indeed, let $M$ be a monoid with identity element $1$.
Take a nonempty set $X$. We now define the category $MX$
in the following way. As set of objects in $MX$ we take $X$.
As morphisms in $MX$ we take all triples $g = (m,x,y)$,
for $x,y \in X$, where we put $d(g) = y$ and $c(g)=x$.
The composition of $g = (m,x,y)$ with $h = (n,y,z)$,
for $m,n \in M$ and $x,y,x \in X$, is defined as
$gh = (mn , x , z)$. Then, for all $x,y \in X$, 
the set $MX(x,y)$ is nonempty and 
$(1,x,x) = (1,x,y)(1,y,x) \in MX(x,y) MX(y,x)$.
Thus, by Proposition \ref{prop:equivalentcategory}(3),
it follows that $MX$ is a hom-set strong category.
On the other hand, it is easily checked that
$MX$ is a groupoid if and only if $M$ is a group.
\end{ex}

\begin{prop}\label{prop:finitecategory}
Suppose $G$ is a hom-set strong category.
Then $G$ is finite if and only if $G_0$
is finite and for every $a \in G_0$ the monoid 
$G(a)$ is finite.
\end{prop}

\begin{proof}
The ``only if'' statement is clear. 
Now we show the ``if'' statement.
Suppose $G_0$ is finite and let all monoids 
$G(a)$, for $a \in G_0$, be finite.
Take different $c,d \in G_0$ with $G(c,d) \neq \emptyset$.
By Proposition \ref{prop:equivalentcategory}(3)
there are $p : c \to d$ and $q : d \to c$
such that $p q = c$. Define maps
$\alpha : G(c,d) \to G(c)$
and $\beta : G(c) \to G(c,d)$ by
$\alpha(g) = qg$, for $g \in G(c,d)$,
and $\beta(h) = p h$, for $h \in G(c)$.
Then $\beta \alpha = {\rm id}_{G(c,d)}$, which,
in particular, implies that $\alpha$ is injective.
Since $G(c)$ is finite this implies that $G(c,d)$ is
finite. Thus, since $G_0$ is finite, we get that
$G_1 = \cup_{a,b \in G_0} G(a,b)$ is finite.
\end{proof}

\section{Hom-set-strongly category graded rings}\label{sec:chaincategory}

In this section, we study the class of object unital
category (and groupoid) graded rings. 
At the end of the section, we prove 
Theorem \ref{thm:generalizationnas}.

\begin{defi}\label{def:categorygraded}
For the rest of this section, we let $S$ denote a
ring which is \emph{$G$-graded}.
Recall from \cite{lundstrom2004,lundstrom2005} 
that this means that there
for every $g \in G_1$ is an additive subgroup 
$S_g$ of $S$ such that $S = \bigoplus_{g \in G_1} S_g$
and for all $g,h \in G_1$, the inclusion 
$S_g S_h \subseteq S_{gh}$ holds,
if $(g,h) \in G_2$, and $S_g S_h = \{ 0 \}$, otherwise.
Note that if $H$ is a subcategory of $G$ then 
$S_H := \oplus_{h \in H_1} S_h$ is a subring of $S$.
Following \cite{cala2021,cala2022}
(see also \cite{oinert2012})
we say that the $G$-grading on $S$ is 
{\it object unital} if 
for all $a \in G_0$ the ring $S_a$ is unital and 
for all $g \in G_1$ and all 
$s \in S_g$ the equalities 
$1_{S_{c(g)}} s = s 1_{S_{d(g)}} = s$ hold.
In that case, $S$ is 
a ring with enough idempotents with
$\{ 1_{S_a} \}_{a \in G_0}$
as a complete set of idempotents and the 
following equality holds for all $a,b \in G_0$:
\begin{equation}\label{eq:1A}
1_{S_a} S 1_{S_b} = S_{G(a,b)}.
\end{equation}
Following \cite{lundstrom2004} we say that
$S$ is \emph{strongly $G$-graded}
if for all $(g,h) \in G_2$, the equality 
$S_g S_h = S_{gh}$ holds.
\end{defi}

\begin{prop}\label{prop:equivalent3}
Suppose $G$ is a hom-set strong category and
let $S$ be an object unital $G$-graded ring.
Then the following properties are equivalent:
\begin{enumerate}

\item $\forall (a,b,c) \in G_0 \times G_0 \times G_0$ 
if two of the groups
$S_{G(a,b)}$, $S_{G(b,c)}$ and $S_{G(a,c)}$ are nonzero,
then the third is also nonzero and 
$S_{G(a,b)} S_{G(b,c)} = S_{G(a,c)}$;

\item $\forall (p,q) \in G_0 \times G_0$ 
if one of the groups $S_G(p,q)$
and $S_{G(q,p)}$ is nonzero, then the other one is also
nonzero and $S_{G(p,q)} S_{G(q,p)} = S_{G(p)}$;

\item  $\forall (p,q) \in I \times I$ 
if one of the groups $S_{G(p,q)}$
and $S_{G(q,p)}$ is nonzero, then the other one is also
nonzero and $1_{S_p} \in S_{G(p,q)} S_{G(q,p)}$;

\item The set $\{ 1_{S_a} \}_{a \in G_0}$
is a strong complete set of idempotents for $S$.

\end{enumerate}
\end{prop}

\begin{proof}
This follows from Proposition \ref{prop:equivalentIII},
Proposition \ref{prop:equivalentcategory} and
equation (\ref{eq:1A}).
\end{proof}

\begin{defi}\label{def:homsetstronglygraded}
Suppose $G$ is a hom-set strong category and
let $S$ be an object unital $G$-graded ring.
If $S$ satisfies any of the equivalent 
properties in Proposition \ref{prop:equivalent3},
then we say that $S$ is \emph{hom-set-strongly} $G$-graded.
\end{defi}

\begin{thm}\label{thm:propertiesagain}
Suppose $S$ is an object unital $G$-graded ring.
\begin{enumerate}

\item If $S$ is left/right artinian (noetherian),
then $G_0$ is finite and, for every $a \in G_0$,
the ring $S_{G(a)}$ is left/right artinian (noetherian).

\item Suppose $G$ is a hom-set strong category and let
$S$ be hom-set-strongly $G$-graded.
Then $S$ is left/right artinian (noetherian) if and only if
$G_0$ is finite and
$S_{G(a)}$ is left/right artinian (noetherian)
for all $a \in G_0$.

\end{enumerate}
\end{thm}

\begin{proof}
This follows from Theorem \ref{thm:main} and
Proposition \ref{prop:equivalent3}.
\end{proof}

\begin{prop}\label{prop:groupoidimplies}
Suppose $S$ is an object unital $G$-graded ring.
If $G$ is a groupoid,
then $S$ is hom-set-strongly $G$-graded.
\end{prop}

\begin{proof}
Suppose $G$ is a groupoid.
By Proposition \ref{prop:groupoidimplieshom}
$G$ is hom-set strong.
Take $p,q \in G_0$. We consider two cases.
Case 1: $S_{G(p,q)} \neq \{ 0\}$.
Then there is $g : q \to p$ with $S_g \neq \{ 0 \}$.
Since $S$ is strongly $G$-graded it 
follows that
$1_{S_p} \in S_p = S_g S_{g^{-1}} \subseteq 
S_{G(p,q)} S_{G(q,p)}$. Thus $S_{G(q,p)} \neq \{ 0 \}$.
Case 2: $S_{G(q,p)} \neq \{ 0 \}$. By a calculation similar
to the one in Case 1 it follows that $S_{G(p,q)} \neq \{ 0 \}$
and that $1_{S_p} \in S_{G(p,q)} S_{G(q,p)}$.
By Proposition~\ref{prop:equivalent3}(3)
it follows that $S$ is hom-set-strongly $G$-graded.
\end{proof}

\begin{defi}\label{def:torsionfree}
Recall that a group is called \emph{torsion-free} 
if the only element in the group of finite order is the 
identity.
More generally, we say that a groupoid $G$ is torsion-free
if for every $a \in G_0$ the group $G(a)$ is torsion-free.
\end{defi}

\begin{thm}\label{thm:nastasescu}
Suppose $H$ is a torsion-free group 
with identity element 1 and let
$T$ be a unital and $H$-graded ring.
Then $T$ is left/right artinian if and only if
$T_1$ is left/right artinian and $T$ is finitely 
generated as a left/right $T_1$-module.  
\end{thm}

\begin{proof}
This is \cite[Cor. 9.6.2]{nastasescu2004} (see also 
\cite[Thm. 1.2]{lannstrom2020}).
\end{proof}

\begin{defi}\label{def:polycyclicbyfinite}
Recall that a group is called 
\emph{polycyclic-by-finite} if it has a finite length 
subnormal series with each factor a finite group 
or an infinite cyclic group.
More generally, we say that a groupoid $G$ is 
polycyclic-by-finite
if for every $a \in G_0$ the group $G(a)$ is 
polycyclic-by-finite.
\end{defi}

\begin{thm}\label{thm:bell}
Suppose $H$ is a polycyclic-by-finite group
with identity 
\linebreak
element 1 and let
$T$ be a unital and strongly $H$-graded ring.
Then $T$ is left/right noetherian
if and only if $T_1$ is left/right noetherian.
\end{thm}

\begin{proof}
See \cite[Prop. 2.5]{bell1987} and
in a slightly more general case 
\cite[Thm. 1.1]{lannstrom2020}
\end{proof}

\subsubsection*{Proof of Theorem \ref{thm:generalizationnas}}
This follows from Theorem \ref{thm:propertiesagain},
Proposition \ref{prop:groupoidimplies},
\linebreak
Theorem \ref{thm:nastasescu} and
Theorem \ref{thm:bell}. \qed

\section{Skew category algebras}\label{sec:chainskew}

In this section, we apply the previous results 
to analyse chain conditions for
skew category algebras which are defined by 
skew category systems of unital rings.
At the end of the section, 
we prove Theorems \ref{thm:generalizationpark}
and \ref{thm:generalizationzelmanov}.

\begin{defi}\label{def:skewcategory}
For the rest of the article, let
$R = \{ R_a \}_{a \in G_0}$ be a collection of 
unital rings and let
$\alpha = \{ \alpha_g : R_{d(g)} 
\to R_{c(g)} \}_{g \in G_1}$ be a collection
of ring isomorphisms.
Following \cite{lundstrom2012} we say that $\alpha$
is a \emph{skew category system} if $\alpha$
is a functor from $G$ to the category of unital rings,
that is, if $\alpha( gh ) = \alpha(g) \alpha(h)$
for all $(g,h) \in G_2$.
Again following \cite{lundstrom2012}, 
we say that the associated 
\emph{skew category algebra}
of $G$ over $R$, denoted by $R *_\alpha G$,
is the set 
of formal finite sums of elements of the form $r g$
for $r \in R_{c(g)}$ and $g \in G_1$.
The addition in $R *_\alpha G$ is defined by the relations
$rg + r' g = (r + r') g$ for $r,r' \in R_{c(g)}$ and $g \in G_1$.
The multiplication in $R *_\alpha G$ is defined by 
the additive extensions
of the relations $r g \cdot r' h = r \alpha_g(r') gh$, for 
$r \in R_{c(g)},r' \in R_{c(h)}$ and $(g,h) \in G_2$, and 
$rg \cdot r'h = 0$, when $g,h \in G_1$ but 
$(g,h) \notin G_2$. 
The ring $R *_\alpha G$ is $G$-graded 
if we put $(R *_\alpha G)_g = R_{c(g)} g$ for 
$g \in G_1$. In fact, with this grading, $R *_\alpha G$
is strongly $G$-graded. Also $R *_\alpha G$ is object unital
since for each $a \in G_0$, the ring 
$(R *_\alpha G)_a = R_a a$ is unital 
with $1_{R_a a} = 1_{R_a} a$. 
If $G$ is a groupoid (group, monoid),
then $R *_\alpha G$ is called a skew groupoid (group, monoid)
algebra of $G$ over $R$.
If all the rings in $R$ coincide with a ring $T$
and the ring isomorphisms in $\alpha$ are identity 
maps, then $R *_\alpha G$ is 
called a \emph{category algebra} of $G$
over $T$ and is denoted by $T[G]$.
In that case, if $G$ is groupoid (group, monoid),
then $T[G]$ is called a groupoid (group, monoid)
algebra of $G$ over $T$.
If $a \in G_0$, then we put
$\alpha(a) = \{ \alpha_g : R_{d(g)} 
\to R_{c(g)} \}_{g \in G(a)}$.
\end{defi}

\begin{prop}\label{prop:impliess}
If $R *_\alpha G$ 
is left/right artinian (noetherian), then 
$G_0$ is finite and, for every $a \in G_0$, 
the skew monoid ring
$R_a *_{\alpha(a)} G(a)$ 
is left/right artinian (noetherian).
\end{prop}

\begin{proof}
This follows from Theorem \ref{thm:propertiesagain}(1).
\end{proof}

\begin{prop}\label{prop:objectagain}
The set $\{ 1_{R_a} a \}_{a \in G_0}$ is a strong
complete set of 
\linebreak
idempotents for $R *_\alpha G$
if and only if the category $G$ is hom-set strong.
In that case, the ring $R *_\alpha G$ is 
hom-set-strongly $G$-graded. 
\end{prop}

\begin{proof}
Put $S = R *_\alpha G$.

Suppose $\{ 1_{R_a} a \}_{a \in G_0}$ is a strong
complete set of idempotents for $S$ and
take $(x,y) \in G_0 \times G_0$ such that 
one of the sets $G(x,y)$ and $G(y,x)$ is nonempty.
\begin{itemize}

\item Case 1: $G(x,y) \neq \emptyset$.
Then $1_{R_x} x S 1_{R_y} y = S_{G(x,y)} \neq \{ 0 \}$.
By Proposition \ref{prop:equivalentIII}(3) it follows
that $S_{G(y,x)} = 1_{R_y} y S 1_{R_x} x \neq \{ 0 \}$
and $1_{R_x} x \in S_{G(x,y)} S_{G(y,x)}$. 
Thus, in particular, $x \in G(x,y)G(y,x)$.

\item Case 2: $G(y,x) \neq \emptyset$. By an argument
similar to the one used in Case 1 it follows that
$G(x,y) \neq \emptyset$ and $y \in G(y,x)G(x,y)$.

\end{itemize}
By Proposition \ref{prop:equivalentcategory}(3)
it follows that $G$ is hom-set strong.

Suppose that $G$ is hom-set strong and
take $(p,q) \in G_0$ such that one of
the additive groups $1_{R_p}p S 1_{R_q} q = S_{G(p,q)}$ 
and $1_{R_q}q S 1_{R_p} p = S_{G(q,p)}$ is nonzero.
\begin{itemize}

\item Case 1: $S_{G(p,q)} \neq \{ 0 \}$.
Then $G(p,q) \neq \emptyset$. Proposition 
\ref{prop:equivalentcategory}(3) implies that 
$G(q,p) \neq \emptyset$ and $p \in G(p,q)G(q,p)$.
Thus $1_{R_p} p \in S_{G(p,q)} S_{G(q,p)}$.

\item Case 2: $S_{G(p,q)} \neq \{ 0 \}$. 
By an argument
similar to the one used in Case 1 it follows that
$G(p,q) \neq \emptyset$ and $q \in G(q,p)G(p,q)$
and $1_{R_q} q \in S_{G(q,p)} S_{G(p,q)}$.

\end{itemize}
By proposition \ref{prop:equivalentIII}(3) 
it follows that 
$\{ 1_{R_a} a \}_{a \in G_0}$ is a strong
complete set of idempotents for $S$.
The last statement follows from Proposition 
\ref{prop:equivalent3}.
\end{proof}

\begin{prop}\label{prop:localskew}
Suppose $G$ is hom-set strong.
Then $R *_\alpha G$ is left/right artinian (noetherian) 
if and only if
$G_0$ is finite and, for every $a \in G_0$, 
the skew monoid ring
$R_a *_{\alpha(a)} G(a)$ 
is left/right artinian (noetherian).
\end{prop}

\begin{proof}
This follows from Theorem \ref{thm:propertiesagain}(2) 
and Proposition~\ref{prop:objectagain}.
\end{proof}

For the rest of the article, $T$ denotes a fixed
unital ring.

\begin{thm}\label{thm:park}
Suppose $H$ is a group and let  
$\alpha : H \to {\rm Aut}(T)$ be a group 
homomorphism. Then the associated skew group ring
$T *_\alpha H$ is left/right artinian if and 
only if $H$ is finite and $T$ is left/right artinian.
\end{thm}

\begin{proof}
This is a result by Park \cite[Thm. 3.3]{park1979}
\end{proof}

\begin{thm}\label{thm:zelmanov}
Suppose $M$ is a monoid.  
Then the monoid algebra 
$T[M]$ is left/right artinian if and only if $M$ is finite
and $T$ is left/right artinian.
\end{thm}

\begin{proof}
This is a result by Zelmanov 
\cite[Corollary at p. 562]{zelmanov1977}.
\end{proof}

\begin{prop}\label{prop:Gfinitenecessary}
If the category algebra $T[G]$ is left/right artinian, 
then $T$ is left/right artinian and $G$ is finite.
\end{prop}

\begin{proof}
This follows from Theorem \ref{thm:zelmanov} and
Propositions \ref{prop:finitecategory} and
\ref{prop:impliess}.
\end{proof}

\subsubsection*{Proof of Theorem \ref{thm:generalizationpark}
and Theorem \ref{thm:generalizationzelmanov}}
This follows from 
Proposition \ref{prop:localskew},
\linebreak
Theorem \ref{thm:park},
Theorem \ref{thm:zelmanov}
and Proposition \ref{prop:Gfinitenecessary}.
\qed


\end{document}